\documentclass[10pt,a4paper]{amsart}%
\usepackage{amsmath}
\usepackage{amsfonts}
\usepackage{amssymb}
\usepackage{hyperref}
\usepackage{enumerate}
\usepackage{graphicx}%
\providecommand{\U}[1]{\protect\rule{.1in}{.1in}}
\newtheorem{theorem}{Theorem}[section]

\newtheorem{corollary}[theorem]{Corollary}

\newtheorem{definition}[theorem]{Definition}

\newtheorem{lemma}[theorem]{Lemma}

\newtheorem{notation}[theorem]{Notation}

\newtheorem{proposition}[theorem]{Proposition}
\newtheorem{remark}[theorem]{Remark}

\begin{document}

\author{Dmitry Gourevitch}
\address{Dmitry Gourevitch, Faculty of Mathematics and Computer Science, Weizmann
Institute of Science, POB 26, Rehovot 76100, Israel }
\email{dimagur@weizmann.ac.il}
\urladdr{\url{http://www.wisdom.weizmann.ac.il/~dimagur}}
\author{Siddhartha Sahi}
\address{Siddhartha Sahi, Department of Mathematics, Rutgers University, Hill Center -
Busch Campus, 110 Frelinghuysen Road Piscataway, NJ 08854-8019, USA}
\email{sahi@math.rugers.edu}
\date{\today}
\title{Intertwining operators between line bundles on Grassmannians}

\begin{abstract}
Let $G=GL(n,F)$ where $F$ is a local field of arbitrary characteristic, and
let $\pi_{1},\pi_{2}$ be representations induced from characters of two
maximal parabolic subgroups $P_{1},P_{2}$. We explicitly determine the space
$Hom_{G}\left(  \pi_{1},\pi_{2}\right)  $ of intertwining operators and prove
that it has dimension $\leq1$ in all cases.

\end{abstract}
\keywords{Reductive group, maximal parabolic, degenerate principal series, derivatives of representations, Radon transform, cosine transform. \\
\indent 2010 MS Classification: 22E50, 44A05, 44A12}
\maketitle

\section{Introduction}

Let $G$ be a reductive group over a local field $F$; then $C^{\infty}\left(
G\right)  $ is a $G\times G$-module with left and right actions given by
$L_{g}f\left(  x\right)  =f\left(  g^{-1}x\right)  $, $R_{g}f\left(  x\right)
=f\left(  xg\right)  $. Let $P\subset G$ be a parabolic subgroup with modular
function $\Delta_{P}$ and let $\chi$ be a character of $P$. The induced
representation $I\left(  P,\chi\right)  $ is the right $G$-action on the space%
\begin{equation}
C^{\infty}\left(  G,P,\chi\right)  :=\left\{  f\in C^{\infty}\left(  G\right)
\mid L_{p^{-1}}f=\chi\left(  p\right)  \Delta_{P}^{1/2}\left(  p\right)
f\text{ for all }p\in P\right\}  , \label{=CG}%
\end{equation}
whose elements may also be regarded as smooth sections of a line bundle on
$G/P$.

We are primarily interested in the group $G=G_{n}:=GL(n,F)$ and its parabolic
subgroups $P=P_{p_{1},p_{2}}$, with $p_{1}+p_{2}=n$, consisting of matrices
$x\in G_{n}$ of the form
\begin{equation}
x=\left[
\begin{array}
[c]{ll}%
x_{11} & x_{12}\\
0 & x_{22}%
\end{array}
\right]  ;x_{ij}\in Mat_{p_{i}\times p_{j}}\text{.} \label{=x}%
\end{equation}
In this case $G/P$ is the Grassmannian of $p_{1}$-dimensional subspaces of
$F^{n}$. The characters of $P$ are of the form $\chi_{1}\otimes\chi_{2}\left(
x\right)  =\chi_{1}\left(  x_{11}\right)  \chi_{2}\left(  x_{22}\right)  $,
where $\chi_{i}$ is a character of $G_{p_{i}}$. Following \cite{BZ-Induced} we
write $\chi_{1}\times\chi_{2}$ instead of $I\left(  P,\chi_{1}\otimes\chi
_{2}\right)  $.\footnote{If $p_{1}$ (resp. $p_{2}$) $=0$ then $P=G$ and
$\chi_{1}\times\chi_{2}=\chi_{2}$ (resp. $\chi_{1}$)$.$}

Let $\pi_{1}=\chi_{1}\times\chi_{2}$ and $\pi_{2}=\chi_{3}\times\chi_{4}$ be
two such representations, where each $\chi_{i}$ is character of $G_{p_{i}}$
with $p_{1}+p_{2}=p_{3}+p_{4}=n$. Our main result is an explicit determination
of the space $Hom_{G_{n}}\left(  \pi_{1},\pi_{2}\right)  $ of intertwining
operators, or \emph{intertwiners} for short; in particular we prove that it
has dimension at most $1$.

It turns out that all intertwiners were previously known. The list includes
such examples as the Radon transform and cosine transform, which are of
considerable geometric interest; indeed these transforms were first
constructed and studied in a geometric context, their intertwining properties
being only recognized much later [\cite{gelfand-graev-rosu}, \cite{alesker}].
One can further supplement the list with two simple examples, scalar operators
for $\pi_{1}=\pi_{2}$, and certain rank $1$ operators obtained as a
composition of two rank $1$ Radon transforms. Finally, for the middle
Grassmannian over archmidean fields ($F=\mathbb{R}$ or $\mathbb{C}$), one also
has discrete families of intertwiners given by certain Capelli-type\emph{
differential} operators. Our main contribution, in addition to the
multiplicity 1 statement, is to show that there are no other intertwiners.

We fix some notation to describe our main results succinctly. For $z\in F$ let
$\nu\left(  z\right)  $ denote the positive scalar by which the additive Haar
measure on $F$ transforms under multiplication by $z$. We will also regard
$\nu$ as a character of $G_{n}$ defined by $\nu(g):=\nu(\det g)$, and we note
that the modular function of $P=P_{p_{1},p_{2}}$ is $\Delta_{P}=\nu^{p_{2}%
}\otimes\nu^{-p_{1}}$. For integers $i\leq j$ we write $[i,j)$ for the
character $\nu^{\frac{i+j}{2}}$of $G_{j-i}$. If $\pi=$ $\chi_{1}\times\chi
_{2}$ then we write $\tilde{\pi}=$ $\chi_{2}\times\chi_{1}$. Finally we write
$\pi_{1}\dashrightarrow\pi_{2}$ to mean that there exists a non-zero
intertwining operator from $\pi_{1}$ to $\pi_{2}$.

\begin{proposition}
\label{std}For any $\pi=$ $\chi_{1}\times\chi_{2}$ we have $\pi\dashrightarrow
\pi$ and $\pi\dashrightarrow\tilde{\pi}$.
\end{proposition}

\begin{proposition}
\label{mix1}Fix an integer $k>0$ and for each integer $0\leq i<k$ define
$\alpha_{i}=[0,i)\times\lbrack i,k)$; then for all integers $0\leq j\neq i<k$
we have%
\[
\widetilde{\alpha_{j}}\dashrightarrow\alpha_{i}\text{ .}%
\]

\end{proposition}

\begin{proposition}
\label{mix2}Fix integers $0<i<j<k$ and define $\beta=[0,j)\times\lbrack i,k)$,
$\gamma=[0,k)\times\lbrack i,j)$, then we have%
\[
\gamma\dashrightarrow\beta,\tilde{\gamma}\dashrightarrow\beta,\tilde{\beta
}\dashrightarrow\gamma,\tilde{\beta}\dashrightarrow\tilde{\gamma}\text{.}%
\]

\end{proposition}

\begin{proposition}
\label{exc1}Fix an integer $k>0$ and let $1$, $\delta$, $\varsigma$ denote the
trivial, $\det$, and $sgn\left(  \det\right)  $ characters of $GL_{k}\left(
\mathbb{R}\right)  ;$ then for all integers $i>0$ we have
\[
1\times\delta^{i}\varsigma\dashrightarrow\delta^{i}\times\varsigma\text{.}%
\]

\end{proposition}

\begin{proposition}
\label{exc2}Fix an integer $k>0$ and let $1$, $\delta$, $\bar{\delta}$ denote
the trivial, $\det$, and $\overline{\det}$ characters of $GL_{k}\left(
\mathbb{C}\right)  ,$ then for all integers $i>0$ and all integers $j$ we have%
\[
1\times\delta^{i}\bar{\delta}^{j}\dashrightarrow\delta^{i}\times\bar{\delta
}^{j},1\times\bar{\delta}^{i}\delta^{j}\dashrightarrow\bar{\delta}^{i}%
\times\delta^{j}\text{ .}%
\]

\end{proposition}

We can get additional instances of $\pi_{1}\dashrightarrow\pi_{2}$ by
considering central twists. To formulate this precisely we introduce the
following notation.

\begin{notation}
\label{X} Given non-negative integers $p_{1}+p_{2}=p_{3}+p_{4}=n$ and
characters $\chi_{i}$ of $G_{p_{i}}$ we write $\mathfrak{X}=\left(  \chi
_{1},\chi_{2},\chi_{3},\chi_{4}\right)  $,
\[
\text{ }H\left(  \mathfrak{X}\right)  =Hom_{G_{n}}\left(  \chi_{1}\times
\chi_{2},\chi_{3}\times\chi_{4}\right)
\]

\end{notation}

We define the\emph{ central twist} of $\mathfrak{X}$ by a character $\psi$ of
$F^{\times}$ to be%
\[
\psi\mathfrak{X=}\left(  \psi\chi_{1},\ldots,\psi\chi_{4}\right)  \text{ with
}\left(  \psi\chi_{i}\right)  \left(  g\right)  =\psi\left(  \det g\right)
\chi_{i}\left(  g\right)  .
\]
It is easy to see that for all $\psi$ we have a natural isomorphism $H\left(
\mathfrak{X}\right)  \approx H\left(  \psi\mathfrak{X}\right)  $.

\begin{definition}
\label{SME}We refer to the $\mathfrak{X}$ obtained by central twists from
Propositions \ref{std} (resp. \ref{mix1},\ref{mix2} ) (resp. \ref{exc1},
\ref{exc2}) as \emph{standard} (resp. \emph{mixed)} (resp.\emph{ exceptional}).
\end{definition}

Our main result is as follows.

\begin{theorem}
\label{thm:main} $\dim H\left(  \mathfrak{X}\right)  \leq1$ with equality iff
$\mathfrak{X}$ is standard, mixed, or exceptional.
\end{theorem}


The intertwiners in the standard case are either scalar operators or
Knapp-Stein operators (cosine transforms). In the mixed case the intertwiners
are rank 1 operators in Proposition \ref{mix1}, and Radon transforms in
Proposition \ref{mix2}. The intertwiners in the exceptional case in
Propositions \ref{exc1} and \ref{exc2} are given by explicit differential
operators. In the real case they factor through a Speh representation and in
the complex case either their domain or their range are irreducible.

In \S \ref{sec:prel} we give some preliminaries on induced representations of
reductive groups. In \S \ref{sec:der} we introduce a key tool: the
Bernstein-Zelevinsky theory of derivatives. This tool is specific for $G_{n}$,
but works uniformly over all fields. In \S \ref{sec:const} we construct the
intertwining operators. In \S \ref{sec:Pf} we finish the proof of Theorem
\ref{thm:main}. The proof is carried out by induction, using the theory of
derivatives and results on finite-dimensional subquotients.

\subsection{Acknowledgements}

We cordially thank Semyon Alesker for posing this question to us (for
$F=\mathbb{R}$) and for useful discussions.



\section{Preliminaries}

\label{sec:prel}

\subsection{Degenerate principal series}

Let $G$ be a reductive group over an arbitrary local field $F$. In this
section we discuss some basic properties of the induced representation
$I\left(  P,\chi\right)  $ on $C^{\infty}\left(  G,P,\chi\right)  $ as in
(\ref{=CG}). For detailed proofs we refer the reader to \cite{BW,Wal1}
and to other standard texts on representation theory.

Let $\mathcal{E}^{\prime}\left(  G\right)  $ denote the set of compactly
supported distributions on $G$, regarded as a left and right $G$-module as
usual via the pairing $\left\langle \cdot,\cdot\right\rangle :\mathcal{E}%
^{\prime}\left(  G\right)  \times C^{\infty}\left(  G\right)  \rightarrow
\mathbb{C}$.

\begin{lemma}
\label{lem: delta1} Let $\varepsilon\in\mathcal{E}^{\prime}\left(  G\right)  $
denote evaluation at $1\in G$, then we have
\[
\left\langle R_{p^{-1}}\varepsilon,f\right\rangle =\chi\left(  p\right)
\Delta_{P}^{1/2}\left(  p\right)  \left\langle \varepsilon,f\right\rangle
\text{ for all }p\in P,f\in C^{\infty}\left(  G,P,\chi\right)
\]

\end{lemma}

\begin{proof}
Indeed both sides are equal to $f\left(  p\right)  $.
\end{proof}

\begin{lemma}
\label{lem: duality} The representations $I\left(  P,\chi\right)  $ and
$I\left(  P,\chi^{-1}\right)  $ are contragredient.
\end{lemma}

\begin{proof}
This is proved in \cite[V.5.2.4]{Wal1}.
\end{proof}

Let $\bar{P}$ denote the parabolic subgroup opposite to $P$. Then the
characters of $P$ and $\bar{P}$ can be identified with those of the common
Levi subgroup $L=P\cap\bar{P}.$

\begin{proposition}
\label{Prop:KnSt}There is a nonzero intertwining operator $I\left(
P,\chi\right)  \rightarrow$ $I\left(  \bar{P},\chi\right)  .$
\end{proposition}

\begin{proof}
See \cite{KnSt} and \cite{Wald} for the archimedean and non-archimedean cases.
\end{proof}

\subsection{Finite dimensional representations}

Let $\left(  \phi,V\right)  $ be an irreducible finite dimensional
representation of a reductive group $G$. We are interested in the possibility
of realizing $\phi$ as a submodule or quotient of some $I\left(
P,\chi\right)  $, which we denote by $\phi\hookrightarrow I\left(
P,\chi\right)  $ and $I\left(  P,\chi\right)  \twoheadrightarrow\phi$
respectively. We start with two simple results.

\begin{lemma}
\label{highest}We have $\dim V^{P,\chi}\leq1$ for all $\chi$, with equality
for at most one $\chi$.
\end{lemma}

\begin{proof}
This is obvious if $\dim V=1$, while $\dim V>1$ only occurs in the archimedean
case, where the result follows from highest weight theory.
\end{proof}

Let $\left(  \phi^{\ast},V^{\ast}\right)  $ be the contragredient
representation of $\left(  \phi,V\right)  $.

\begin{lemma}
\label{lem:FinSubFrob} We have $\phi\hookrightarrow I\left(  P,\chi\right)  $
(resp. $I\left(  P,\chi\right)  \twoheadrightarrow\phi$) iff $\left(  V^{\ast
}\right)  ^{P,\chi^{-1}\Delta_{P}^{-1/2}}$ (resp. $V^{P,\chi\Delta_{P}^{-1/2}%
}$) is nonzero. For a given $P$, there is at most one such $\chi$ in each case.
\end{lemma}

\begin{proof}
If $\phi\hookrightarrow I\left(  P,\chi\right)  $ then by Lemma
\ref{lem: delta1} the restriction $\varepsilon|_{V}$ gives an element in
$\left(  V^{\ast}\right)  ^{P,\chi^{-1}\Delta_{P}^{-1/2}}$, easily seen to be
nonzero. Conversely the matrix coefficient with respect to such an element
provides an imbedding $\phi\hookrightarrow I\left(  P,\chi\right)  $. Next by
Lemma \ref{lem: duality}, we see that
\[
I\left(  P,\chi^{-1}\right)  \twoheadrightarrow\phi^{\ast}\iff\phi
\hookrightarrow I\left(  P,\chi\right)  \iff\left(  V^{\ast}\right)
^{P,\chi^{-1}\Delta_{P}^{-1/2}}\neq0.
\]
Replacing $\phi$ by $\phi^{\ast}$ and $\chi$ by $\chi^{-1}$, we deduce
$I\left(  P,\chi\right)  \twoheadrightarrow\phi\iff$ $V^{P,\chi\Delta
_{P}^{-1/2}}\neq0$.

The second part of the Lemma follows from Lemma \ref{highest}.
\end{proof}

\begin{proposition}
\label{prop:Uniq} If $\phi\hookrightarrow I\left(  P,\chi\right)  $ (resp.
$I\left(  P,\chi\right)  \twoheadrightarrow\phi$) then $\phi$ is the unique
irreducible submodule (resp. quotient) of $I\left(  P,\chi\right)  $.
\end{proposition}

\begin{proof}
By Lemma \ref{lem: duality} it suffices to deal with that case $\phi
\hookrightarrow I\left(  P,\chi\right)  $. If $P$ is minimal, then the result
follows from the Langlands classification (\cite[Ch IV and Ch XI]{BW}), once
we note that $\phi$ is a Langlands submodule of $I\left(  P,\chi\right)  $.
Otherwise choose a minimal parabolic $P_{0}\subset P$. Then we have
\[
\phi\hookrightarrow I\left(  P,\chi\right)  \subset I\left(  P_{0},\chi
_{0}\right)  \text{ with }\chi_{0}=\Delta_{P_{0}}^{-1/2}\left(  \chi\Delta
_{P}^{1/2}\right)  |_{P_{0}};
\]
but $\phi$ is the unique submodule of $I\left(  P_{0},\chi_{0}\right)  $,
hence also of $I\left(  P,\chi\right)  $.
\end{proof}

Fix a minimal parabolic subgroup $P_{0}\subset G$ and let $\mathcal{M}$ be the
set of pairs $\left(  P,\chi\right)  $ such that $P$ is a \emph{maximal}
parabolic containing $P_{0}$ and $\chi$ is a character of $P$.

\begin{lemma}
\label{lem:Uniq} If $\dim V>1$ then $\phi\hookrightarrow I\left(
P,\chi\right)  $ (resp. $I\left(  P,\chi\right)  \twoheadrightarrow\phi$) for
at most one $\left(  P,\chi\right)  \in\mathcal{M}$.
\end{lemma}

\begin{proof}
By Lemma \ref{lem:FinSubFrob} it suffices to show that if $\left(  V^{\ast
}\right)  ^{P_{1},\chi_{1}},\left(  V^{\ast}\right)  ^{P_{2},\chi_{2}}\neq0$
for $\left(  P_{i},\chi_{i}\right)  \in\mathcal{M}$ then $P_{1}=P_{2}$. By the
Lemmas \ref{highest} and \ref{lem:FinSubFrob} we conclude that $\chi_{1}$,
$\chi_{2}$ have the same restriction $\chi_{0}$ (say) to $P_{0}$, and that%
\[
\left(  V^{\ast}\right)  ^{P_{1},\chi_{1}}=\left(  V^{\ast}\right)
^{P_{0},\chi_{0}}=\left(  V^{\ast}\right)  ^{P_{2},\chi_{2}}%
\]

If $P_{1},P_{2}$ were\emph{ different} maximal parabolic subgroups then they
would generate $G$, and the one-dimensional space $\left(  V^{\ast}\right)
^{P_{0},\chi_{0}}$ would be $G$-invariant, contradicting the assumption that
$V$, and hence $V^{\ast}$, is irreducible of dimension $>1$.
\end{proof}

\subsection{Intertwining differential operators}

In this subsection we suppose that $F$ is an archimedean field. Let $G$ be a
real reductive group, let $P=LN$ be a parabolic subgroup and denote the
opposite nilradical by $\bar{N}$. We denote the Lie algebras of $G,\bar{N}$
etc. by $\mathfrak{g,\bar{n}}$ etc. and their enveloping algebras by
$\mathcal{U}\left(  \mathfrak{g}\right)  $, $\mathcal{U}\left(  \mathfrak{\bar
{n}}\right)  $ etc. The left and right $G$-actions on $C^{\infty}\left(
G\right)  $ give rise to vector fields $L_{X},R_{X}$ for $X\in\mathfrak{g}$,
and more generally to differential operators $L_{u}$,$R_{u}$ for
$u\in\mathcal{U}\left(  \mathfrak{g}\right)  $.

We are interested in triples $\left(  u,\chi,\eta\right)  $ where
$u\in\mathcal{U}\left(  \mathfrak{\bar{n}}\right)  $ and $\chi,\eta$ are
characters of $P$ such that $L_{u}$ maps the space $C^{\infty}\left(
G,P,\chi\right)  $ to $C^{\infty}\left(  G,P,\eta\right)  $. Since left and
right actions commute, such an $L_{u}$ is automatically an intertwining
differential operator between the induced representations $I\left(
P,\chi\right)  $ and $I\left(  P,\eta\right)  $, and we will refer to $\left(
u,\chi,\eta\right)  $ as an \emph{intertwining triple}.

\begin{proposition}
\label{prop:KV}Suppose (a) $\mathfrak{n}$ is abelian, (b) $u$ transforms by
the character $\chi\eta^{-1}$ under the adjoint action of $L$, and (c) the
product $\chi\eta$ extends to a character of $G$; then $\left(  u,\chi
,\eta\right)  $ is an intertwining triple
\end{proposition}

\begin{proof}
If $\eta=\chi^{-1}$ then this is proved in \cite[Proposition 2.3]{KV}, and the
same proof works for the general case.
\end{proof}

\begin{remark}
\label{rem: KV} In the context of Proposition \ref{prop:KV}, since
$\mathfrak{\bar{n}}$ is abelian, we may identify $\mathcal{U}\left(
\mathfrak{\bar{n}}\right)  $ with the symmetric algebra $\mathcal{S}\left(
\mathfrak{\bar{n}}\right)  $. Furthermore, we may identify $\mathfrak{\bar{n}%
}$ with $\mathfrak{n}^{\ast}$ and thus regard $u\in\mathcal{S}\left(
\mathfrak{\bar{n}}\right)  $ as polynomial function on $\mathfrak{n}$.
\end{remark}

\section{Derivatives}

\label{sec:der}


If $\chi$ is a character of $G_{p}$ with $p>0$, we write $\chi^{\prime}$ for
its restriction to $G_{p-1}$ and $\chi^{+}$ for its extension to $G_{p+1}$
(\textit{i.e.} the unique character such that $\left(  \chi^{+}\right)
^{\prime}=\chi$). If $\mathfrak{X}=\left(  \chi_{1},\chi_{2},\chi_{3},\chi
_{4}\right)  $ with all $p_{i}>0$ then we define%
\[
\mathfrak{X}^{\prime}=\left(  \chi_{1}^{\prime},\chi_{2}^{\prime},\chi
_{3}^{\prime},\chi_{4}^{\prime}\right)  ,\mathfrak{X}^{+}=\left(  \chi_{1}%
^{+},\chi_{2}^{+},\chi_{3}^{+},\chi_{4}^{+}\right)  .
\]

\begin{lemma}
\label{lem:plus} If all $p_{i}>1$ and $\mathfrak{X}^{\prime}$ is\ standard,
mixed or exceptional then so is $\mathfrak{X}$.
\end{lemma}

\begin{proof}
Since all $p_{i}>1$ we have $\mathfrak{X=}\left(  \mathfrak{X}^{\prime
}\right)  ^{+}$. The result is obvious if $\mathfrak{X}^{\prime}$ is standard
or exceptional since $\delta_{p}^{\prime}=\delta_{p-1}$ etc. For the mixed
case we note that if $i<j$ then
\[
\chi=[i,j)\implies\nu^{1/2}\chi^{+}=\nu^{\frac{i+j+1}{2}}=[i,j+1).
\]
Now writing $\sim$ to denote equality up to a (common) central twist, we see
that%
\[
\mathfrak{X}^{\prime}\sim\left(  \lbrack i_{1},j_{1}),\cdots,[i_{4}%
,j_{4})\right)  \implies\mathfrak{X}\sim\left(  \lbrack i_{1},j_{1}%
+1),\cdots,[i_{4},j_{4}+1)\right)
\]
It follows easily that if $\mathfrak{X}^{\prime}$ is mixed then, up to a
twist, $\mathfrak{X}$ is as in Lemma \ref{mix2}.
\end{proof}

We will prove the main result (Theorem \ref{thm:main}) by induction on $n$,
using ideas from \cite{BZ-Induced, AGS}. We refer the reader to those papers
for the notion of \emph{depth} for an admissible representation of $G_{n},$
and for the definition of the functor $\Phi$ which maps admissible
representations of $G_{n}$ of  depth $\leq2$ to admissible
representations of $G_{n-2}$. In \cite{BZ-Induced} this functor is denoted $\Phi^-$.

\begin{proposition}
(\cite{BZ-Induced, AGS})

\begin{enumerate}
\item $\Phi$ is an exact functor and $\Phi(\chi_{i}\times\chi_{j})=\chi
_{i}^{\prime}\times\chi_{j}^{\prime}$ if $p_{i},p_{j}>1.$

\item Every subquotient of $\chi_{i}\times\chi_{j}$ has depth $\leq2$

\item If $\pi$ has depth $1$ then $\pi$ is finite dimensional and $\Phi(\pi)=0$.

\item If $\pi$ has depth $2$ then $\Phi\left(  \pi\right)  \neq0$
\end{enumerate}
\end{proposition}

Let $H(\mathfrak{X})$ be as in Notation \ref{X} and we let $H_{0}\left(
\mathfrak{X}\right)  \subset H\left(  \mathfrak{X}\right)  $ denote the
subspace of finite rank operators.

\begin{corollary}
\label{cor:X'}If $H_{0}\left(  \mathfrak{X}\right)  =0$ then $\Phi$ defines  an
imbedding $H\left(  \mathfrak{X}\right)  \hookrightarrow$ $H\left(
\mathfrak{X}^{\prime}\right)  $.
\end{corollary}


\section{Construction of intertwining operators}

\label{sec:const}

We now prove Propositions \ref{std} -- \ref{exc2}. As before we write $\pi
_{1}\dashrightarrow\pi_{2}$ if there exists a non-zero intertwining operator
from $\pi_{1}$ to $\pi_{2}$.

\begin{proof}
[Proof of Proposition \ref{std}]The identity operator gives $\pi
\dashrightarrow\pi$. Next we write
\[
P=P_{p_{1},p_{2}},\chi=\chi_{1}\otimes\chi_{2},\tilde{P}=P_{p_{2},p_{1}%
},\tilde{\chi}=\chi_{2}\otimes\chi_{1}.
\]
Then we have $\pi=I\left(  P,\chi\right)  $ and $\tilde{\pi}=I\left(
\tilde{P},\tilde{\chi}\right)  \approx I\left(  \bar{P},\chi\right)  $ since
$\left(  \tilde{P},\tilde{\chi}\right)  $ and $\left(  \bar{P},\chi\right)  $
are $G$-conjugate. Now we get $\pi\dashrightarrow\tilde{\pi}$ from Proposition
\ref{Prop:KnSt}.
\end{proof}

\begin{proof}
[Proof of Proposition \ref{mix1}]It follows from Lemma \ref{lem:FinSubFrob}
that $\phi\hookrightarrow\alpha_{i}$ and $\widetilde{\alpha_{j}}%
\twoheadrightarrow\phi$, where $\phi$ is the character $\nu^{k/2}$ of $G_{k}$.
Thus we get a non-zero map $\widetilde{\alpha_{j}}\rightarrow\phi
\rightarrow\alpha_{i}$.
\end{proof}

\begin{proof}
[Proof of Proposition \ref{mix2}]By Proposition \ref{mix1} and induction by
stages we get maps%
\begin{align*}
\gamma\rightarrow\lbrack0,j)\times\lbrack j,k)\times\lbrack i,j)\rightarrow
\beta,  &  \text{ }\tilde{\gamma}\rightarrow\lbrack i,j)\times\lbrack
0,i)\times\lbrack i,k)\rightarrow\beta,\\
\tilde{\beta}\rightarrow\lbrack i,k)\times\lbrack0,i)\times\lbrack
i,j)\rightarrow\gamma,  &  \text{ }\tilde{\beta}\rightarrow\lbrack
0,j)\times\lbrack j,k)\times\lbrack i,j)\rightarrow\tilde{\gamma}.
\end{align*}
To see that the composite maps are non-zero, we note that each map is non-zero
on the one-dimensional space of vectors fixed by the maximal compact subgroup.
\end{proof}

\begin{proof}
[Proof of Proposition \ref{exc1}]Let $G=GL_{2k}\left(  \mathbb{R}\right)  $
and $P=P_{k,k}$ then $\mathfrak{n}\approx Mat_{k\times k}\left(
\mathbb{R}\right)  $ is abelian. Let $u\in U\left(  \mathfrak{\bar{n}}\right)
$ correspond, as in Remark \ref{rem: KV}, to the polynomial function $\det
^{i}$ on $\mathfrak{n}$, and set%
\[
\chi=1\otimes\delta^{i}\varsigma,\eta=\delta^{i}\otimes\varsigma.
\]
Then $u$ transforms by the character $\chi\eta^{-1}=\delta^{-i}\otimes
\delta^{i}$ under the adjoint action of $L=G_{k}\times G_{k}$, and the product
$\chi\eta=\delta^{i}\otimes\delta^{i}$ extends to the character $\delta^{i}$
of $G=G_{2k}$. Thus $\left(  u,\chi,\eta\right)  $ is an intertwining triple
by Proposition \ref{prop:KV}, and the result follows.
\end{proof}

\begin{proof}
[Proof of Proposition \ref{exc2}]This is proved similarly, using the
polynomial functions $\det^{i}$ and $\overline{\det}^{i}$ on $\mathfrak{n}%
\approx Mat_{k\times k}\left(  \mathbb{C}\right)  .$
\end{proof}

\begin{remark}
\label{rem:ex}

In Proposition \ref{exc1} the maps factor through the Speh representation (see
\cite{SaSt}). In Proposition \ref{exc2}, either the source or the target of
the map are irreducible (see \cite{HL}).
\end{remark}

\section{Proof of Theorem \ref{thm:main}}

\label{sec:Pf}

Let $H_{0}\left(  \mathfrak{X}\right)  \subset H(\mathfrak{X})$ denote the
subspace of maps of finite rank. If $H_{0}\left(  \mathfrak{X}\right)  \neq0$,
then there is a finite-dimensional representation $\phi$ that is a quotient of
$\chi_{1}\times\chi_{2}$ and a submodule of $\chi_{3}\times\chi_{4}$. We will
indicate this by writing $\phi\vdash\mathfrak{X.}$

\begin{proposition}
\label{prop:Unique} We have $\dim H(\mathfrak{X})\leq1$.
\end{proposition}

\begin{proof}
First suppose $H_{0}\left(  \mathfrak{X}\right)  \neq0$, and let $\phi
\vdash\mathfrak{X}$ be as above. By Proposition \ref{prop:Uniq} $\phi$ is the
unique irreducible quotient of $\chi_{1}\times\chi_{2}$ and the unique
irreducible submodule of $\chi_{3}\times\chi_{4}$. It follows that any map in
$H(\mathfrak{X})$ factors through $\phi$ and hence $\dim H(\mathfrak{X})=1$.

If $H_{0}\left(  \mathfrak{X}\right)  =0$ then by Corollary \ref{cor:X'} we get
an imbedding $H\left(  \mathfrak{X}\right)  \hookrightarrow H(\mathfrak{X}%
^{\prime})$. The result now follows by induction on $n=p_{1}+p_{2}=p_{3}%
+p_{4}$, with the initial cases $n=0$ and $n=1$ being trivial.
\end{proof}

\begin{lemma}
\label{lem: SM} If $H_{0}\left(  \mathfrak{X}\right)  \neq0$ then
$\mathfrak{X}$ is standard or mixed.
\end{lemma}

\begin{proof}
Let $\phi\vdash\mathfrak{X}$ be as above. If $\dim\phi=1$, then Lemma
\ref{lem:FinSubFrob} implies that up to a central twist by $\phi$ we have
\[
\chi_{1}\times\chi_{2}=[ j,n) \times[ 0,j) \text{ and }\chi_{3}\times\chi
_{4}=[ 0,i) \times[ i,n) \text{ for some }0\leq i,j<n\text{.}%
\]
Thus if $i=j$ then $\mathfrak{X}$ is standard, otherwise $\mathfrak{X}$ is
mixed as in Lemma \ref{mix1}.

If $\dim\phi>1$ then $F$ is archimedean and by Lemma \ref{lem:FinSubFrob}
$\phi$ is a submodule of $\chi_{2} \times\chi_{1}$. By Lemma \ref{lem:Uniq}
$\phi$ is a submodule of a unique degenerate principal series. Thus $\chi
_{3}=\chi_{2}$ and $\chi_{4}=\chi_{1}$ and hence $\mathfrak{X}$ is standard.
\end{proof}

Before proving the next result we make some simple observations.

\begin{lemma}
Let $\mathfrak{X=}\left(  \chi_{1},\chi_{2},\chi_{3},\chi_{4}\right)  $ with
$H(\mathfrak{X})\neq0$, and write $\chi_{i}=\psi_{i}\circ\delta_{p_{i}}$then%
\begin{equation}
\psi_{1}\left(  z\right)  ^{p_{1}}\psi_{2}\left(  z\right)  ^{p_{2}}=\psi
_{3}\left(  z\right)  ^{p_{3}}\psi_{4}\left(  z\right)  ^{p_{4}}\text{ for all
}z\in F^{\times}\text{.} \label{=psi}%
\end{equation}

\end{lemma}

\begin{proof}
It follows from the definition of induction that the central element
$zI_{n}\in G_{n}$ acts on $\chi_{1}\times\chi_{2}$ and $\chi_{3}\times\chi
_{4}$ by the scalars $\psi_{1}\left(  z\right)  ^{p_{1}}\psi_{2}\left(
z\right)  ^{p_{2}}$ and $\psi_{3}\left(  z\right)  ^{p_{3}}\psi_{4}\left(
z\right)  ^{p_{4}}$. If $Hom_{G_{n}}(\chi_{1}\times\chi_{2},\chi_{3}\times
\chi_{4})\neq0$ then these scalars must be the same.
\end{proof}

\begin{corollary}
\label{cor:one}Let $\mathfrak{X=}\left(  \chi_{1},\chi_{2},\chi_{3},\chi
_{4}\right)  $ with $H(\mathfrak{X})\neq0$.

\begin{enumerate}
\item[(i)] If $p_{1}=1$ then $\chi_{1}$ is uniquely determined by $\chi
_{2},\chi_{3},\chi_{4}$.

\item[(ii)] If $p_{1}=p_{3}=1$ and $\chi_{2}=\chi_{4}$, then $\chi_{1}%
=\chi_{3}$.
\end{enumerate}
\end{corollary}

\begin{proof}
For case (i) we note that $\psi_{1}\left(  z\right)  =$ $\psi_{2}\left(
z\right)  ^{-p_{2}}\psi_{3}\left(  z\right)  ^{p_{3}}\psi_{4}\left(  z\right)
^{p_{4}}$ by (\ref{=psi}). In case (ii) we have $p_{2}=p_{4}=n-1$ and
$\psi_{2}=\psi_{4}$, hence by(\ref{=psi}) we get $\psi_{1}=\psi_{3}$.
\end{proof}

\begin{proposition}
\label{prop:class} If $H(\mathfrak{X})\neq0$ then $\mathfrak{X}$ is standard,
mixed, or exceptional.
\end{proposition}

\begin{proof}
We proceed by induction on $n=p_{1}+p_{2}=p_{3}+p_{4}$. The case $n=1$ is
obvious and the case $n=2$ follows from standard facts about principal series
for $GL_{2}$ (see e.g. \cite[\S \S 5.6,5.7]{Wal1} for the archimedean case).
Thus from now on we may assume that $\,n\geq3$ and that the result is true for
$n-2$. By the previous lemma we may also assume that $H_{0}\left(
\mathfrak{X}\right)  =0$. In particular we may assume that each $p_{i}>0$ and
by induction that $\mathfrak{X}^{\prime}$ is standard, mixed, or exceptional.

Let $I$ be the set of indices $i$ such that $p_{i}=1$. Since $n\geq3$, $I$ can
contain at most one index from each of the sets $\left\{  1,2\right\}  $ and
$\left\{  3,4\right\}  $. If $I=\emptyset$ then the result follows from Lemma
\ref{lem:plus}.

Suppose $\left\vert I\right\vert =1$. If $I=\left\{  1\right\}  $ then up to a
central twist we have%
\[
H\left(  \mathfrak{X}^{\prime}\right)  =Hom\left(  [ 0,n-2) ,[ 0,p_{3}-1)
\times[ p_{3}-1,n-2) \right)  ,
\]
and hence we get $\mathfrak{X}\sim\left(  \chi_{1},[ 0,n-1) ,[ 0,p_{3}) ,[
p_{3}-1,n-1) \right)  $. By Corollary \ref{cor:one} we must have $\chi_{1}=[
p_{3}-1,p_{3}) $, and hence $\mathfrak{X}$ is mixed. The proof is similar if
$I=\left\{  2\right\}  ,\left\{  3\right\}  $ or $\left\{  4\right\}  .$

Suppose $\left\vert I\right\vert =2$. If $I=\left\{  1,3\right\}  $ then up to
a central twist we get%
\[
H\left(  \mathfrak{X}^{\prime}\right)  =Hom\left(  [0,n-2),[0,n-2)\right)  .
\]
Thus we have $\mathfrak{X}\sim\left(  \chi_{1},[0,n-1),\chi_{3}%
,[0,n-1)\right)  $. By Corollary \ref{cor:one} we get $\chi_{1}=\chi_{3}$ and
hence $\mathfrak{X}$ is standard. The proof is similar in the other cases with
$\left\vert I\right\vert =2.$
\end{proof}

\begin{proof}
[Proof of Theorem \ref{thm:main}]This follows from Propositions \ref{std} --
\ref{exc2}, \ref{prop:Unique} and \ref{prop:class}.
\end{proof}

\end{document}